\documentclass[11pt,a4paper]{article}
\usepackage{amsmath,amsfonts,amssymb,latexsym,graphics,epsfig,url}
\usepackage{color}
\usepackage{amsthm,enumerate}
\usepackage[english]{babel}
\usepackage{graphicx}
\input amssym.def
\newsymbol\rtimes 226F
\newfont{\nset}{msbm10}
\newcommand{\ns}[1]{\mbox{\nset #1}}

\newcommand{\executeiffilenewer}[3]{%
\ifnum\pdfstrcmp{\pdffilemoddate{#1}}%
{\pdffilemoddate{#2}}>0%
{\immediate\write18{#3}}\fi%
}

\newtheorem{theorem}{Theorem}[section]

\newtheorem{proposition}[theorem]{Proposition}

\newtheorem{definition}[theorem]{Definition}

\DeclareMathOperator{\dgr}{dgr}
\DeclareMathOperator{\tr}{tr}
\DeclareMathOperator{\spec}{sp}

\def\R{\ns R}
\def\vecv{\mbox{\boldmath $v$}}

\def\vec0{\mbox{\boldmath $0$}}

\def\A{\mbox{\boldmath $A$}}
\def\B{\mbox{\boldmath $B$}}

\def\G{\Gamma}

\def\I{\mbox{\boldmath $I$}}

\def\S{\mbox{\boldmath $S$}}

\def\I{\mbox{\boldmath $I$}}

\def\vec0{\mbox{\bf 0}}
\def\ev{\mathop{\rm ev}\nolimits}

\def\tr{\mathop{\rm tr}\nolimits}

\def\G{\Gamma}

\def\Re{\mathbb R}

\begin{document}

\title{A new class of polynomials from the spectrum of a graph, and its application to bound the $k$-independence number}
\author{M. A. Fiol \\
{\small Departament de Matem\`atiques} \\
{\small Universitat Polit\`ecnica de Catalunya, Barcelona, Catalonia} \\
{\small Barcelona Graduate School of Mathematics}\\
{\small{\tt{miguel.angel.fiol@upc.edu}}}}

\maketitle

\begin{abstract}
The $k$-independence number of a graph is the maximum size of a set of vertices at pairwise distance greater than $k$. A graph is called $k$-partially walk-regular if the number of closed walks of a given length $l\le k$, rooted at a vertex $v$, only depends on $l$. In particular, a distance-regular graph is also $k$-partially walk-regular for any $k$.
In this note, we introduce a new family of polynomials obtained from the spectrum of a graph.
These polynomials, together with the interlacing technique, allow us to give tight spectral bounds on the $k$-independence number of a $k$-partially walk-regular graph. Together with some examples where the bounds are tight, we also show that the odd graph $O_{\ell}$ with $\ell$ odd has no $1$-perfect code.
\end{abstract}

\maketitle

\noindent{\em keywords:} Graph,  $k$-independence number, spectrum, interlacing, %Regular partition,
minor polynomial,
$k$-partially walk-regular, %Antipodal distance-regular graph, Diameter.

\noindent{\em Mathematics Subject Classifications:} 05C50, 05C69.

\section{Introduction}
Given a graph $G$, let $\alpha_k = \alpha_k(G)$ denote the size of the largest set of vertices such that any two vertices in the set are at distance larger than $k$. Thus, with this notation, $\alpha_1$ is just the independence number $\alpha$ of a graph.
The parameter $\alpha_k(G)$ therefore represents the largest number of vertices which can be $k+1$ spread out in $G$. It is known that determining $\alpha_k$ is NP-Hard in general \cite{kZ1993}.

The $k$-independence number of a graph is directly related to other combinatorial parameters such as the average distance \cite{FH1997}, packing chromatic number \cite{GHHHR2008}, injective chromatic number \cite{HkSS2002}, and strong chromatic index \cite{M2000}. Upper bounds on the $k$-independence number directly give lower bounds on the corresponding distance or packing chromatic number \cite{AM2002}, as well as necessary conditions for the existence of perfect codes.

In this note we generalize and improve the known spectral upper bounds for the $k$-independence number from \cite{Fiolkindep},  \cite{act16} and \cite{acf19}. For some cases, we also show that our bounds are sharp.                                                                                                               Let $G=(V,E)$ be a graph with $n=|V|$ vertices, $m=|E|$ edges, and adjacency matrix $\A$ with  spectrum
$
\spec G=\{\theta_0,\theta_1^{m_1},\ldots,\theta_d^{m_d}\}.
$
where the different eigenvalues are in decreasing order, $\theta_0>\theta_1>\cdots>\theta_d$, and the superscripts stand for their multiplicities.
When the eigenvalues are presented with possible repetitions, we shall indicate them by
$
\ev G:  \lambda_1 \geq \lambda_2 \geq \cdots\geq \lambda_n.
$

The first known spectral bound for the independence number $\alpha$ of a graph is due to Cvetkovi\'c \cite{c71}.

\begin{theorem}[Cvetkovi\'c \cite{c71}]
\label{thm:cvetkovic}
Let $G$ be a graph with eigenvalues $\lambda_1\ge \cdots \ge \lambda_n$. Then,
$$
\alpha\le \min \{|\{i : \lambda_i\ge 0\}| , |\{i : \lambda_i\le 0\}|\}.
$$
\end{theorem}

Another well-known result is the following bound due to Hoffman (unpublished; see for instance Haemers  \cite{h95}).

\begin{theorem}[Hoffman \cite{h95}]
\label{thm:hoffman}
If $G$ is a regular graph on $n$ vertices with eigenvalues $\lambda_1\ge \cdots \ge \lambda_n$, then
\[
\alpha \leq n\frac{-\lambda_n }{\lambda_1 - \lambda_n}.
\]
\end{theorem}

Regarding the $k$-independence number, the following three results are known. The first is due to Fiol \cite{Fiolkindep} and requires a preliminary definition. Let $G$ be a graph with distinct eigenvalues $\theta_0 > \cdots > \theta_d$. Let $P_k(x)$ be chosen among all polynomials $p(x) \in \Re_k(x)$, that is, polynomials of real coefficients and degree at most $k$, satisfying $|p(\theta_i)| \leq 1$ for all $i = 1,...,d$, and such that $P_k(\theta_0)$ is maximized. The polynomial $P_k(x)$ defined above is called the {\em $k$-alternating polynomial} of $G$ and  was shown to be unique in \cite{fgy96}, where it was used to study the relationship between the spectrum of a graph and its diameter.

\begin{theorem}[Fiol \cite{Fiolkindep}]
\label{thm:fiol}
Let $G$ be a $d$-regular graph on $n$ vertices, with distinct eigenvalues $\theta_0 >\cdots > \theta_d$ and let $P_k(x)$ be its $k$-alternating polynomial. Then,
\[
\alpha_k \leq \frac{2n}{P_k(\theta_0) + 1}.
\]
\end{theorem}
More recently, Cvetkovi\'c-like and Hoffman-like bounds were given by Abiad, Cioab\u{a}, and Tait in \cite{act16}. 	
\begin{theorem}[Abiad, Cioab\u{a},  Tait \cite{act16}]
\label{previous1act}
Let $G$ be a graph on $n$ vertices with adjacency matrix $\A$, with eigenvalues $\lambda_1 \geq \cdots \geq \lambda_n$. Let $w_k$ and $W_k$ be respectively the smallest and the largest diagonal entries of $\A^k$. Then,
\[
\alpha_k \leq \min\{|\{i : \lambda_i^k \geq w_k(G)\}| ,  |\{i : \lambda_i^k \leq W_k(G)\}|\}.
\]
\end{theorem}

\begin{theorem}[Abiad, Cioab\u{a},  Tait \cite{act16}]
%\label{previous2act}
\label{thm:abiad}
Let $G$ be a $\delta$-regular graph on $n$ vertices with adjacency matrix $\A$, whose distinct eigenvalues are $\theta_0(=\delta) > \cdots> \theta_d$. Let $\widetilde{W_k}$ be the largest diagonal entry of $\A+\A^2+\cdots+\A^k$. Let $\theta = \max\{|\theta_1| , |\theta_d|\}$. Then,
\[
\label{aida}
\alpha_k \leq n \frac{\widetilde{W_k}+ \sum_{j = 1}^k \theta^j}{\sum_{j = 1}^k \delta^j + \sum_{j = 1}^k\theta^j}.
\]
\end{theorem}
	
Finally, as a consequence of a generalization of the last two theorems,
Abiad, Coutinho, and the author \cite{acf19}, proved the following results.

\begin{theorem}[Abiad, Coutinho, Fiol \cite{acf19}]
\label{theo-gen-k}
Let $G$ be a $\delta$-regular graph with $n$ vertices and distinct eigenvalues $\theta_0(=\delta)>\theta_1> \cdots > \theta_d$. Let $W_k=W(p)=\max_{u\in V}\{\sum_{i=1}^k(\A^k)_{uu}\}$.
Then, the $k$-independence number of $G$ satisfies the following:
\begin{itemize}
\item[$(i)$]
If $k=2$, then
\begin{equation*}
\alpha_2\le n\frac{\theta_0+\theta_i\theta_{i-1}}{(\theta_0-\theta_i)
(\theta_0-\theta_{i-1})},
\end{equation*}
where $\theta_i$ is the largest eigenvalue not greater than $-1$.
\item[$(ii)$]
If $k>2$ is odd, then
\begin{equation*}
\label{eq-k}
\alpha_k(G)\le n\frac{W_k-\sum_{j=0}^k \theta_d^j}{\sum_{j=0}^k \delta^j-\sum_{j=0}^k \theta_d^j}.
\end{equation*}
\item[$(iii)$]
If $k>2$ is even, then
\begin{equation*}
\alpha_k(G)\le n\frac{W_k+1/2}{\sum_{j=0}^k \delta^j+1/2}.
\end{equation*}
\item[$(iv)$]
If $G=(V,E)$ is a walk-regular graph, then
\begin{equation*}
\alpha_k(G)\le n\frac{1-\lambda(q_k)}{q_k(\delta)-\lambda(q_k)}
\label{coro-walk-reg}
\end{equation*}
for $k=0,\ldots,d-1$, where $q_k=p_0+\cdots+p_k$ with the $p_i$'s being the predistance polynomials of $G$ (see next section), and $\lambda(q_k)=\min_{i\in [2,d]}\{q_k(\theta_i)\}$.
\end{itemize}
\end{theorem}

\section{Some Background}
For basic notation and results see \cite{biggs,g93}. Let $G=(V,E)$ be a (simple) graph with $n=|V|$ vertices, $m=|E|$ edges, and adjacency matrix $\A$ with  spectrum $\spec G=\{\theta_0,\theta_1^{m_1},\ldots,\theta_d^{m_d}\}$. When the eigenvalues are presented with possible repetitions, we shall indicate them by $\lambda_1 \geq \lambda_2 \geq \cdots\geq \lambda_n$.
Let us consider the scalar product in $\Re_d[x]$:
$$
\langle f,g\rangle_G=\frac{1}{n}\tr(f(\A)g(\A))=\frac{1}{n}\sum_{i=0}^{d} m_i f(\theta_i)g(\theta_i).
$$
The so-called {\em predistance polynomials} $p_0(=1),p_1,\ldots, p_d$ are a sequence of orthogonal polynomials with respect to the above product, with  $\dgr p_i=i$, and they are normalized in such a way that $\|p_i\|_G^2=p_i(\theta_0)$ (this makes sense since it is known that $p_i(\theta_0)>0$) for $i=0,\ldots,d$. Therefore they are uniquely determined, for instance, following the Gram-Schmidt process. These polynomials were introduced by Fiol and Garriga in \cite{fg97} to prove the so-called `spectral excess theorem' for distance-regular graphs, where $p_0(=1),p_1,\ldots, p_d$ coincide with the so-called distance polynomials .
See \cite{cffg09} for further details and applications.

A graph $G$ is called {\em $k$-partially walk-regular}, for some integer $k\ge 0$, if the number of closed walks of a given length $l\le k$, rooted at a vertex $v$, only depends on $l$. Thus, every (simple) graph is $k$-partially walk-regular for $k=0,1$, and every regular graph is $2$-partially walk-regular. Moreover $G$ is $k$-partially walk-regular for any $k$ if and only if $G$ is walk-regular, a concept introduced by Godsil and Mckay in \cite{gm80}.  For example, it is well-known that every distance-regular graph is walk-regular (but the converse does not hold).

%\subsection{Eigenvalue interlacing}
Eigenvalue interlacing is a powerful and old technique that has found countless applications in combinatorics and other fields. This technique will be used in several of our proofs. For more details, historical remarks, and other applications, see Fiol and Haemers \cite{f99,h95}.
Given square matrices $\A$ and $\B$ with respective eigenvalues $\lambda_1\geq \cdots \geq \lambda_n$ and $\mu_1 \geq \cdots \geq \mu_m$, with $m<n$, we say that the second sequence {\em interlaces} the first one if, for all $i = 1,\ldots,m$, it follows that
$\lambda_i \geq \mu_i \geq \lambda_{n-m+i}$.

\begin{theorem}[Interlacing \cite{f99,h95}]
\label{theo-interlacing}
Let $\S$ be a real $n \times m$ matrix such that $\S^T \S = \I$, and let $\A$ be a $n \times n$ matrix with eigenvalues $\lambda_1 \geq \cdots \geq \lambda_n$. Define $\B = \S^T \A \S$, and call its eigenvalues $\mu_1 \geq\cdots \geq \mu_m$. Then,
\begin{enumerate}[(i)]
\item
The eigenvalues of $\B$ interlace those of $\A$.
\item
If $\mu_i = \lambda_i$ or $\mu_i = \lambda_{n-m+i}$, then there is an eigenvector $\vecv$ of $\B$ for $\mu_i$ such that $\S \vecv$ is eigenvector of $\A$ for $\mu_i$.
\item
If there is an integer $k \in \{0,\ldots,m\}$ such that $\lambda_i = \mu_i$ for $1 \leq i \leq k$,  and $\mu_i = \lambda_{n-m+i}$ for $ k+1 \leq i \leq m$ $(${\em tight interlacing}$)$,  then $\S \B = \A \S$.
\end{enumerate}
\end{theorem}

Two interesting particular cases where interlacing occurs (obtained by choosing
appropriately the matrix $\S$) are the following. Let $\A$ be the adjacency matrix of a graph $G=(V,E)$. First, if $\B$ is a principal submatrix of $\A$, then $\B$ corresponds to the adjacency matrix of an induced subgraph
$G'$ of $G$. Second, when, for a given partition of the vertices of $\G$, say $V=U_1\cup\cdots\cup U_m$, $\B$ is the so-called {\em quotient matrix} of $\A$, with elements
$b_{ij}$, $i,j=1,\ldots,m$, being the average row sums of the corresponding block $\A_{ij}$ of $\A$. Actually, the quotient matrix $\B$ does not need to be
symmetric or equal to $\S^\top\A\S$, but in this case $\B$ is
similar to (and therefore has the same spectrum as) $\S^\top\A\S$.
%Moreover, if the interlacing is tight,
%Theorem~\ref{theo-interlacing}$(iii)$ reflects that $\S$ corresponds to a {\em
%regular} (or {\em equitable}) partition of $\A$, that is, each
%block of the partition has constant row and column sums. Then
%the bipartite induced subgraphs $G_{ij}$, with adjacency matrices $\A_{ij}$, for $i\neq i$, are biregular, and the subgraphs
%$G_{ii}$ are regular.

\section{The minor polynomials}
In this section we introduce a new class of polynomials, obtained from the different eigenvalues of a graph, which are used later to derive our main results.

%The motivation for considering such polynomials is the following:
Let  $G$ be a $k$-partially walk-regular graph with adjacency matrix $\A$ and  spectrum $\spec G=\{\theta_0,\theta_1^{m_1},\ldots,\theta_d^{m_d}\}$.
Let $p$ be a  polynomial  of degree at most $k$, satisfying $p(\theta_0)=1$ and $p(\theta_i)\ge 0$ for $i=1,\ldots,d$. Then, in Section \ref{main-section} we prove that $G$ has  $k$-independence number satisfying  the bound $\alpha_k\le \tr p_k(\A)=\sum_{i=0}^d m_i p_k(\theta_i)$.
So, the search for the best result motivates the following definition.
\begin{definition}
\label{def-minor-p}
Let $G=(V,E)$ be a graph with  adjacency matrix $\A$ with  spectrum $\spec G=\{\theta_0,\theta_1^{m_1},\ldots,\theta_d^{m_d}\}$.
For a given $k=0,1,\ldots,d$, let us consider the set of real polynomials ${\cal P}_k=\{p\in \R_k:p(\theta_0)=1, p(\theta_i)\ge 0, 1\le i\le d\}$, and the continuous function $\Psi: {\cal P}_k\rightarrow \R^+$ defined by $\Psi(p)=\tr p(\A)$. Then, the $k$-minor polynomial of $G$ is the point
$p_k$ where $\Psi$ attains its minimum:
$$
\tr p_k(\A)=\min\left\{\tr p(\A) : p\in {\cal P}_k \right\}.
$$
\end{definition}

%Let $G=(V,E)$ be a graph with $n=|V|$ vertices and adjacency matrix $\A$ with  spectrum $\spec G=\{\theta_0,\theta_1^{m_1},\ldots,\theta_d^{m_d}\}$. Let $\R_k[x]$ denote the $(k+1)$-dimensional vector space of polynomials of degree at most $k$.
An alternative approach to the $k$-minor polynomials is the following: Let $p_k$ be the polynomial defined by $p_k(\theta_0)=x_0=1$ and $p_k(\theta_i)=x_i$, for $i=1,\ldots,d$,  where the vector $(x_1,x_2,\ldots,x_d)$ is a solution of  the following linear programming problem:

\begin{center}
\frame{
 $\begin{array}{rl}
 & \\
 {\tt minimize} & \sum_{i=0}^d m_ix_i\\
 {\tt with\ constraints} & f[\theta_0,\ldots,\theta_m]=0,\ m=k+1,\ldots,d\\
                       & x_i\ge 0,\ i=1,\ldots,d,\\
                        &
 \end{array}$}
 \end{center}
 %where the variables are $x_0(=1),x_1,\ldots,x_d$ correspond to the values of the searched polynomial $p\in {\cal P}_k$ at the points $\theta_0,\theta_1,\ldots,\theta_d$, respectively, and
 where $f[\theta_0,\ldots,\theta_m]$ denote the $m$-th divided differences of Newton interpolation, recursively defined by  $f[\theta_i,\ldots,\theta_j]=\frac{f[\theta_{i+1},\ldots,\theta_j]-f[\theta_i,\ldots,\theta_{j-1}]}
{\theta_j-\theta_{i}}$, where $j>i$, starting with $f[\theta_i]=p(\theta_i)=x_i$, $0\le i\le d$.

Thus, we can easily compute the minor polynomial, for instance by using  the simplex method. Moreover, as the problem is in the so-called standard form, with
$d$ variables, $x_1,\ldots,x_d$, and $d-(k+1)+1=d-k$
equations, the `basic vectors' have at least $d-(d-k)=k$ zeros. Note also that, from the conditions of the programming problem, the $k$-minor polynomial turns out to be of the form
$p_k(x)=f[\theta_0]+f[\theta_0,\theta_1](x-\theta_0)+\cdots +f[\theta_0,\ldots,\theta_k](x-\theta_0)\cdots (x-\theta_{k-1})$, with degree at most $k$. Consequently when we apply the simplex method, we obtain a $k$-minor polynomial $p_k$ with degree $k$ and exactly $k$ zeros at the mesh $\theta_1,\ldots,\theta_d$.
In fact, as shown in the following lemma, a $k$-minor polynomial has exactly $k$ zeros in the interval $[\theta_d,\theta_0)$.

\begin{proposition}
\label{theo:pols}
Let $G$ be a graph with  spectrum $\spec G=\{\theta_0,\theta_1^{m_1},\ldots,\theta_d^{m_d}\}$. Then, for every $k=0,1,\ldots,d$, a $k$-minor polynomial $p_k$ has degree $k$ with its $k$ zeros in $[\theta_d,\theta_0)\subset \Re$.
\end{proposition}

\begin{proof}
We only need to deal with the case $k\ge 1$.
Assume that a  $k$-minor polynomial $p_k$ has the zeros $\xi_{r}\le \xi_{r-1}\le \cdots \le \xi_{1}$ with $r\le k$. Then, it can be written as $p_k(x)=\prod_{i=1}^r \frac{x-\xi_i}{\theta_0-\xi_i}$. Let us first show that $\xi_r\ge \theta_d$. By contradiction, assume that $\xi_r<\theta_d$, and let $\theta_i$ the smallest eigenvalue which is not a zero of $p_k$ (the existence of such a $\theta_i$ is guaranteed from the condition $r\le k$). Then, the polynomial $q_r(x)=\frac{x-\theta_i}{\theta_0-\theta_i}\prod_{j=1}^{r-1}\frac{x-\xi_i}{\theta_0-\xi_i}$, with degree $r\le k$ satisfies the conditions $q_k(\theta_0)=1$, $q_k(\theta_i)\ge 0$ for $i=1,\ldots,d$, and $\Psi(q_k)<\Psi(p_k)$ since
$\frac{\theta_j-\theta_i}{\theta_0-\theta_i}<\frac{\theta_j-\xi_r}{\theta_0-\xi_r}$ for $j>i$, a contradiction with the fact that $\Psi(p_k)$ is minimum. Second, let us prove, again by contradiction, that $\xi_1>\theta_0$. Otherwise, we could consider the polynomial $q_{k-1}$, with degree $r-1<k$, defined as $q_{k-1}(x)=\prod_{i=2}^r \frac{x-\theta_i}{\theta_0-\theta_i}$ satisfying again $q_{k-1}(\theta_0)=1$, and $q_{k-1}k(\theta_i)\ge 0$ for $i=1,\ldots,d-1$ since $\frac{\theta_i-\xi_1}{\theta_0-\xi_1}>1$ for all $i=1,\ldots,d$. But, from the same inequalities, we also have $\Psi(q_{k-1})<\Psi(p_k)$, a contradiction.

Finally, assume that $r<k$. Since all zeros $\xi_r\le \cdots \le \xi_r$ are in the interval $[\theta_d,\theta_0)$, we can consider again the smallest one $\theta_i$ such that $p_k(\theta_i)>0$. Then, reasoning as before, the polynomial
$q_{r+1}(x)=\frac{x-\theta_i}{\theta_0-\theta_i}\prod_{i=1}^r\frac{x-\xi_i}{\theta_0-\xi}$, with degree $r+1\le k$ leads to the desired contradiction $\Psi(q_{r+1})<\Psi(p_k)$.
\end{proof}

The above results, together with $p_k(\theta_0)=1$ and $p_k(\theta_i)\ge 0$ for $i=1,\ldots,d$ drastically reduce the number of possible candidates for $p_k$. Let us consider some particular values of $k$:

\begin{itemize}
\item
The cases $k=0$ and $k=d$ are easy. Clearly, $p_0=1$, and $p_d$ has zeros at all the points $\theta_i$ for $i\neq 0$. In fact, $p_d=\frac{1}{n}H$, where $H$ is the Hoffman polynomial \cite{hof63}.
\item
For $k=1$, the only zero of $p_1$ must be at $\theta_d$. Hence,
\begin{equation}
\label{p1}
p_1(x)=\frac{x-\theta_d}{\theta_0-\theta_d}.
\end{equation}
Moreover, since $p_1(\theta_i)< 1$ for every $i=1,\ldots,d$, we have that
$$
(1=)\Psi(p_{d})<\Psi(p_{d-1})<\Psi(p_{d-2})<\cdots < \Psi(p_1)<  \Psi(p_0)(=n),
$$
since, for $k=0,\ldots,d-1$, $p_{k+1}(\theta_i)\le p_kp_1(\theta_i)< p_k(\theta_i)$ for every $i=1,\ldots,d$.
\item
For $k=2$, the two zeros of $p_2$ must be at consecutive eigenvalues $\theta_{i}$ and $\theta_{i-1}$. More precisely, the same reasonings used in \cite{acf19} shows that $\theta_i$ must be the largest eigenvalue not greater than $-1$. Then, with these values,
\begin{equation}
\label{p2}
p_2(x)=\frac{(x-\theta_i)(x-\theta_{i-1})}{(\theta_0-\theta_i)(\theta_0-\theta_{i-1})}.
\end{equation}
\item
When $k=3$, the only possible zeros of $p_3$ are $\theta_d$ and the consecutive pair $\theta_i$, $\theta_{i-1}$ for some $i\in [2,d-1]$. In this case, empirical results seem to point out that such a pair must be around the `center' of the mesh (see the examples below).
\item
When $k=d-1$, the polynomial $p_{d-1}$ takes only one non-zero value at the mesh, say at $\theta$, which seems to be located at one of the `extremes' of the mesh. In fact, when  $G$ is an $r$-antipodal distance-regular graph, we show in the last section that either $\theta=\theta_1$ or $\theta=\theta_d$ for odd $d$  yields the tight bound (that is, $r$) for $\alpha_{d-1}$, as does Theorem \ref{thm:fiol}.
Consequently, for such a graph with odd $d$, we have two different $(d-1)$-minor polynomials, say $p$ and $q$, and hence infinitely many $(d-1)$-minor polynomials of the form $r=\gamma p+(1-\gamma)q$ where $\gamma\in [0,1]$. (Notice that, if $\gamma\not\in\{0,1\}$, then $r$ must have some zero not belonging to the mesh $\{\theta_1,\ldots,\theta_d\}$.)
\end{itemize}

Now, let us give all the $k$-minor polynomials, with $k=1,\ldots,d$, for two particular distance-regular graphs. Namely, the Hamming graph $H(2,7)$ and the Johnson graph  $J(14,7)$ (for more details about these graphs, see for instance \cite{bcn89}).
First, we recall that the Hamming graph $H(2,7)$ has spectrum
$$
\spec H(2,7)=\{7^1,5^7,3^{21},1^{35},-1^{35},-3^{21},-5^7,-7^1\}.
$$
Then, the different minor polynomials $p_0,\ldots,p_7$ are shown in Figure \ref{fig1}, and their values $x_i=p_k(\theta_i)$ at the different eigenvalues $\theta_0,\ldots,\theta_7$  are shown in Table \ref{table2}.

\begin{figure}[h!]
\begin{center}
\includegraphics[width=13cm]{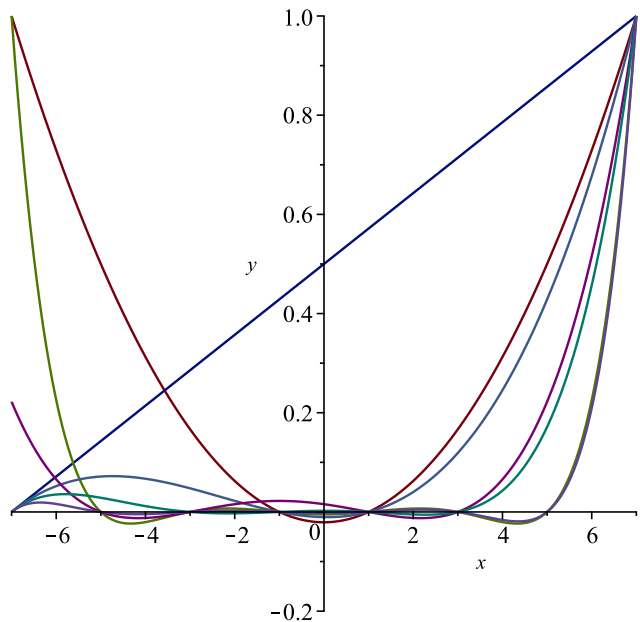}
\vskip-7.5cm
\caption{The polynomials of the Hamming graph $H(2,7)$.}
\label{fig1}
\end{center}
\end{figure}

	\begin{table}[h!]
		\begin{center}
			\begin{tabular}{|c|ccccccc|c|}
				\hline
				$k$ & $x_7$  & $x_6$ & $x_5$ & $x_4$ & $x_3$ & $x_2$ & $x_1$ & $x_0$\\
				\hline
				$1$ & 0 & 1/7 & 2/7 & 3/7 & 4/7 & 5/7 & 6/7 & 1 \\
				\hline
				$2$ & 1 & 1/2 & 1/6 & 0 & 0 & 1/6 & 1/2 & 1 \\
				\hline
				$3$  & 0  & 1/14 & 1/21 & 0 & 0 & 5/42 & 3/7 & 1 \\
				\hline
				$4$  & 2/9 & 0 & 0  & 1/45 & 0 & 0 & 2/9 & 1 \\
				\hline
				$5$ & 0 & 1/35 & 0 & 0 & 0 & 0 & 6/35 & 1 \\
				\hline
				$6$ & 1 & 0 & 0 & 0 & 0 & 0 & 0 & 1 \\
				\hline
				$7$ & 0 & 0 & 0 & 0 & 0 & 0 & 0 & 1 \\
				\hline
			\end{tabular}
		\end{center}
		\caption{Values $x_i=p_k(\theta_i)$ of the $k$-minor polynomials of the Hamming graph $H(2,7)$.}
		\label{table2}
	\end{table}

As another example, consider the
the Johnson graph $J(14,7)$ (see, for instance, \cite{bcn89,g93}). This is an antipodal (but not bipartite) distance-regular graph, with $n=3432$ vertices, diameter $D=7$, and spectrum
$$
\spec J(14,7)=\{49^1, 35^{13}, 23^{77}, 13^{273}, 5^{637}, -1^{1001}, -5^{1001}, -7^{429}\}.
$$
Then the solutions of the linear programming problem are in Table \ref{table3}, which correspond to the minor polynomials shown in Figure \ref{fig2}

\begin{table}[h!]
		\begin{center}
			\begin{tabular}{|c|ccccccc|c|}
				\hline
				$k$ & $x_7$  & $x_6$ & $x_5$ & $x_4$ & $x_3$ & $x_2$ & $x_1$ & $x_0$\\
				\hline
				$1$ & 0 & 1/28 & 3/28 & 3/14 & 5/15 & 15/28 & 3/4 & 1 \\
				\hline
				$2$ & 9/275 & 1/55 & 0 & 0 & 14/275 & 54/275 & 27/55 & 1 \\
				\hline
				$3$  & 0  & 5/1232 & 1/176 & 0 & 0 & 75/1232 & 5/16 & 1 \\
				\hline
				$4$  & 1/1485 & 0 & 0  & 0 & 0 & 14/495 & 2/9 & 1 \\
				\hline
				$5$ & 0 & 1/2860 & 0 & 0 & 0 & 0 & 27/260 & 1 \\
				\hline
				$6$ & 0 & 0 & 0 & 0 & 0 & 0 & 1/13 & 1 \\
				\hline
				$7$ & 0 & 0 & 0 & 0 & 0 & 0 & 0 & 1 \\
				\hline
			\end{tabular}
		\end{center}
		\caption{Values $x_i=p_k(\theta_i)$ of the $k$-minor polynomials of the Johnson graph $J(14,7)$.}
		\label{table3}
	\end{table}

	\begin{figure}[h!]
		\begin{center}
			\includegraphics[width=13cm]{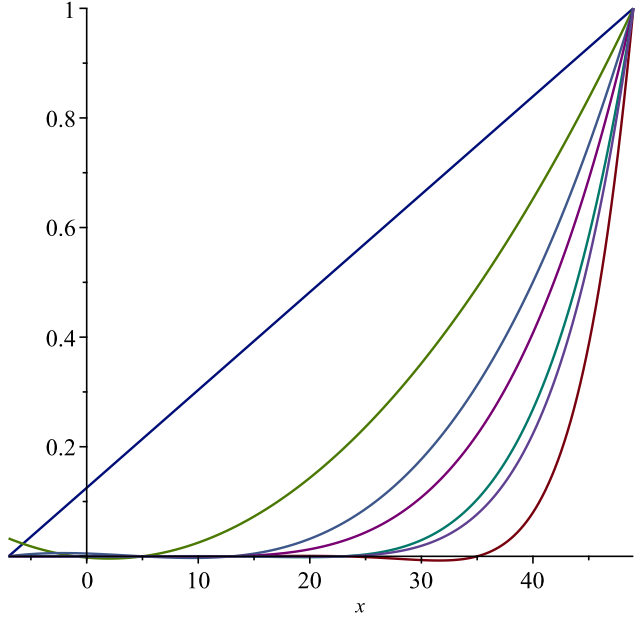}
			\vskip-7.6cm
			\caption{The polynomials of the Johnson graph $J(14,7)$.}
			\label{fig2}
		\end{center}
	\end{figure}
%%%%%%%%%%%%%%%%%%%%%%%%%%%%%%%%%%%%%%%%%%%%%%%%%%%%%%%%%%%%%%%%

\section{A tight bound for the $k$-independence number}
\label{main-section}
Now we are ready to derive our main result about the $k$-independent number of a $k$-partially walk-regular graph. The proof is based on the interlacing technique.
\begin{theorem}
\label{new-theo}
Let $G$ be a $k$-partially walk-regular graph with $n$ vertices, adjacency matrix $\A$, and spectrum $\spec G=\{\theta_0^{m_0},\ldots,\theta_d^{m_d}\}$.
Let  $p_k\in \Re_k[x]$ be a $k$-minor polynomial. Then, for every $k=0,\ldots,d-1$, the $k$-independence number $\alpha_k$ of $G$ satisfies
\begin{equation}
\label{eq:thm2}
\alpha_k\le \tr p_k(\A)=\sum_{i=0}^d m_i p_k(\theta_i).
\end{equation}
\end{theorem}
\begin{proof}
 Let $U$ be a $k$-independent set of $G$ with $r=|U|=\alpha_k(G)$ vertices. Again, assume the first columns (and rows) of $\A$ correspond to the vertices in $U$. Consider the partition of said columns according to $U$ and its complement. Let $\S$ be the normalized characteristic matrix of this partition. The quotient matrix of $p(\A)$ with regards to this partition is given by
\begin{align}
\label{B_k=2}
\S^T p(\A) \S = \B_k & =
\left(
\begin{array}{cc}
\frac{1}{r}\sum_{u\in U}(p_k(\A))_{uu} & p_k(\theta_0)-\frac{1}{r}\sum_{u\in U}(p_k(\A))_{uu}\\
\frac{r p_k(\theta_0)-\sum_{u\in U}(p(\A))_{uu}}{n-r}  & p_k(\theta_0)-\frac{r p_k(\theta_0)-\sum_{u\in U}(p(\A))_{uu}}{n-r}
\end{array}
\right)\\
 &=\left(
 \begin{array}{cc}
 \frac{1}{n}\sum_{i=0}^d m_i p_k(\theta_i) & 1-\frac{1}{n}\sum_{i=0}^d m_i p_k(\theta_i)\\
 \frac{r-\frac{r}{n}\sum_{i=0}^d m_i p_k(\theta_i)}{n-r}  & 1-\frac{r -\frac{r}{n}\sum_{i=0}^d m_i p_k(\theta_i)}{n-r}
 \end{array}
 \right),
\end{align}
with eigenvalues $\mu_1=p(\lambda_1)=1$ and
$$
\mu_2=\tr \B_k-1=w(p_k)-\frac{r -rw(p_k)}{n-r}.
$$
where $w(p_k)=\frac{1}{n}\sum_{i=0}^d m_i p_k(\theta_i)$.
Then, by interlacing, we have
\begin{equation}\frac{{2\ell\choose \ell}}{2(\ell+1)}\frac{{2\ell\choose \ell}}{2(\ell+1)}\frac{{2\ell\choose \ell}}{2(\ell+1)}
\label{interlacing:theo2}
0\le \mu_2\le w(p_k)-\frac{r-rw(p_k)}{n-r},
\end{equation}
whence, solving for $r$, we get $r\le nw(p_k)$ and the result follows.
\end{proof}

As mentioned in the previous section, notice that, in fact, the proof works for any polynomial $p$ satisfying $p(\theta_0)=1$ and $p(\theta_i)\ge 0$ for $i=1,\ldots,d$. By way of example, if $G$ is a distance-regular graph with distance polynomials $p_0,\ldots,p_d$, we could take $p(x)=\frac{q_k^2(x)}{q_k^2(\theta_0)}$, with degree $2k$, where the sum polynomial $q_k=p_0+\cdots +p_k$ satisfies
$\|q_k\|_G^2=q_k(\theta_0)$. Now, recall that $q_k(\theta_0)=n_{k}$ corresponds to   the number  of vertices at distance $\le k$ from any vertex of $G$ (see, for instance Biggs \cite{biggs}). 
Thus, we obtain
$$
\alpha_{2k}\le \Psi(p)=\sum_{i=0}^d m_i\frac{q_k^2(\theta_i)}{q_k^2(\theta_0)}=\frac{n}{q_k^2(\theta_0)}\|q_k\|_G^2=\frac{n}{n_k},
$$
as expected. 

Another possibility is to use the polynomial $p(x)=\frac{P_k(x)+1}{P_k(\theta_0)+1}$, where $P_k$ is the $k$-alternating polynomial.
In this case, when $G$ is an $r$-antipodal distance-regular graphs and $k=d-1$, it turns out that the $d$-distance polynomial is
$p_d=H-\frac{r}{2}P_{d-1}+\frac{r}{2}-1$, where $H$ is the Hoffman polynomial (see \cite{Fiolkindep}). Then, we get
$\Psi(p)=\frac{2n}{P_{d-1}(\theta_0)+1}$, which coincides with the bound for $\alpha_{d-1}$ given in Theorem  \ref{thm:fiol}.

Let us now consider some particular cases of  Theorem \ref{new-theo} by using the minor polynomials.

\subsubsection*{The case $k=1$.}
As mentioned above, $\alpha_1$ coincides with the standard independence number $\alpha$. In this case the minor polynomial is
 $p_1(x)=\frac{x-\theta_d}{\theta_0-\theta_d}$. Then, \eqref{eq:thm2} gives
\begin{equation}
\label{eq-k=1}
\alpha_1=\alpha\le \tr p_1(\A)=\frac{-n\theta_d}{\theta_0-\theta_d},
\end{equation}
which is Hoffman's bound in Theorem \ref{thm:hoffman}.

\subsubsection*{The case $k=2$.}
We already stated that $p_2(x)=\frac{(x-\theta_i)(x-\theta_{i-1})}{(\theta_0-\theta_i)(\theta_0-\theta_{i-1})}$.
Then, \eqref{eq:thm2} yields
\begin{equation}
\label{eq-k=2}
\alpha_2\le \tr p_2(\A)=n\frac{\theta_0+\theta_i\theta_{i-1}}{(\theta_0-\theta_i)
(\theta_0-\theta_{i-1})},
\end{equation}
in agreement with the result  of \cite{acf19} (here in Theorem \label{theo-gen-k}$(i)$). Moreover, in the same paper, two infinite families of (distance-regular) graphs where the bound \eqref{eq-k=2} is tight were provided.

\subsection*{Some examples}
To compare the above bounds with those  obtained in \cite{Fiolkindep} and \cite{act16} (here in Theorems \ref{thm:fiol} and \ref{thm:abiad}, respectively), let us consider again the Hamming graph $H(2,7)$ and the Johnson graph $J(14,7)$. Thus, in Table \ref{table4} we show the bounds obtained for $\alpha_k(H(2,7))$, whereas those of  $\alpha_k(J(14,7))$ are shown in Table \ref{table5}. (Recall that every distance-regular graph is also walk-regular.)
%%%%%%%%%%%%%%
\begin{table}[h!]
\begin{center}
\begin{tabular}{|c|ccccccc| }
\hline\hline
$k$ & $1$ & $2$ & $3$  & $4$ & $5$ & $6$ & $7$ \\
\hline
Bound from Theorem \ref{thm:fiol} & 109 & 72 & 36  & 19 & 7 & 2 & -- \\
\hline
Bound from Theorem \ref{thm:abiad} ($k> 2$) & - & - & 65 & 67  & 64 & 65 & 64 \\
\hline
Bound from Theorem \ref{theo-gen-k}$(i)$-$(iii)$ & - & 21 & 56 & 6  & 55 & 3 & 55 \\
\hline
Bound from Theorem \ref{new-theo} &
				\bf 64 & \bf 16 & \bf 8 & \bf 3 & \bf 2 & \bf 2  & \bf 1 \\
				\hline\hline
			\end{tabular}
		\end{center}
		\caption{Comparison of the bounds for $\alpha_k$ in the Hamming graph $H(2,7)$.}
		\label{table4}
	\end{table}
Note that, in general, the bounds obtained by Theorem \ref{new-theo} constitute a significant improvement with respect to those in \cite{Fiolkindep,act16}. In particular, the bounds for $k=6,7$ are equal to the correct values $\alpha_6=2$ (since both graphs are $2$-antipodal), and $\alpha_7=1$ (since their diameter is $D=7$). Besides notice that, in the case of the Hamming graph, $\alpha_2=16$ since it contains the perfect Hamming code $H(7,4)$.

\begin{table}[h!]
\begin{center}
\begin{tabular}{|c|ccccc| }
\hline\hline
$k$ & $3$  & $4$ & $5$ & $6$ & $7$ \\
\hline
%$P_k(\theta_0)$ & 464 & 125 & 20 & 2 & -- \\
%\hline
%$W_k$ ($p=x+\cdots+x^k$) & 637 & 17150 & 469910 & 15193479 & 537790827 \\
%\hline
%$\theta$  & $35$  & $35$ & $35$ & $35$ & $35$ \\
%\hline
%$\lambda(p)$ ($p=x+\cdots+x^k$) & -301 & 0 & -14707 & 0 & -720601 \\
%\hline
%$q_k(\delta)$ & 1716 & 2941 & 3382 & 3431 & 3432 \\
%\hline
%$\lambda(q_k)$ & -40 & -75 & -24 & -1 & 0 \\
%\hline
Bound from Theorem \ref{thm:fiol} & 464 & 125 & 20 & 2  & -- \\
\hline
Bound from Theorem \ref{thm:abiad} & 935 & 721 & 546 & 408  & 302 \\
\hline
Bound from Theorem \ref{theo-gen-k}$(ii)$-$(iii)$ & 26 & 10 & 5 & 3  & 2 \\
\hline
Bound from Theorem \ref{theo-gen-k}$(iv)$ & 80 & 86 & 25 & 2  & 1 \\
\hline
Bound from Theorem \ref{new-theo} & \bf 19 & \bf 6 & \bf 2 & \bf 2  & \bf 1 \\
\hline\hline
\end{tabular}
\end{center}
\caption{Comparison of bounds for $\alpha_k$ in the Johnson graph $J(14,7)$.}
\label{table5}
\end{table}

\subsection{Antipodal distance-regular graphs}
\label{antipodal}
Finally, we consider an infinite family where our bound for $\alpha_{d-1}$ is tight.
With this aim, we assume that the minor polynomial takes non-zero value only at $\theta_1$.
Thus, $p_{d-1}(x)=\frac{1}{\prod_{i=2}^d (\theta_0-\theta_i)}\prod_{i=2}^d (x-\theta_i)$.
Then, the bound \eqref{eq:thm2} of Theorem \ref{new-theo} is
$$
\sum_{i=0}^d m_i p_{d-1}(\theta_i)=m_0p_{d-1}(\theta_0)+m_{1}p_{d-1}(\theta_1)=1+m_1\frac{\prod_{i=2}^d (\theta_1-\theta_i)}{\prod_{i=2}^d (\theta_0-\theta_i)}=1+m_1\frac{\pi_1}{\pi_0}
$$
where, in general, $\pi_i=\prod_{j=0,j\neq i}|\theta_i-\theta_j|$ for $i\in [0,d]$.
Now suppose that $G$ is an $r$-antipodal distance-regular graph.
Then, in  \cite{Fiolkindep} it was shown that $G$ is so
if and only if its eigenvalue multiplicities are
$m_i=\pi_0/\pi_i$ for $i$ even, and $m_i=(r-1)\pi_0/\pi_i$ for $i$ odd.
So, with $m_1=(r-1)\pi_0/\pi_1$, we get
$$
\alpha_{d-1}\le 1+m_1\frac{\pi_1}{\pi_0}=r,
$$
which is the correct value.

When $G$  is an $r$-antipodal distance-regular graph with odd $d$, we can also consider the minor polynomial $q_{d-1}$ which takes non-zero value only at $\theta_d$, that is  $q_{d-1}(x)=\frac{1}{\prod_{i=1}^{d-1} (\theta_0-\theta_i)}\prod_{i=1}^{d-1} (x-\theta_i)$. Then, reasoning as above,
we get again the tight bound $\Psi(q_{d-1})=1+m_d\frac{\pi_d}{\pi_0}=r$.

\subsection{Odd graphs}
For every integer $\ell\ge 2$, the  odd graphs $O_\ell$ constitute a well-known family of distance-regular graphs with interactions between graph theory and other areas of combinatorics, such as coding theory and design theory.
The vertices of $O_\ell$ correspond to the $(\ell-1)$-subsets of a $(2\ell-1)$-set, and adjacency is defined by void intersection. In particular, $O_3$ is the Petersen graph. In general, the odd $O_\ell$ is a $\ell$-regular graph with order
$n={2\ell-1\choose \ell-1}=\frac{1}{2}{2\ell\choose \ell}$, diameter $\ell-1$, and its eigenvalues and multiplicities are $\theta_i=(-1)^i (\ell-i)$ and
$m(\theta_i)={2\ell-1\choose i}-{2\ell-1\choose i-1}$ for $i=0,1,\ldots,\ell-1$.  For more details, see for instance, Biggs \cite{biggs} and Godsil  \cite{g93}.

In Table \ref{table5}  we show the bounds of the $k$-independence numbers for $O_{\ell}$, $\ell=2,3,4,5$ given by Theorem \ref{new-theo}. The numbers in bold faces, $7$ and $66$, correspond to the known $1$-perfect codes in $O_4$ and $O_6$, respectively.
\begin{table}[h!]
\begin{center}
\begin{tabular}{|c|cccc| }
\hline\hline
graph / $k$ & $2$  & $3$ & $4$ & $5$  \\
\hline
$O_4$ & \bf 7 & -- & -- &  --   \\
\hline
$O_5$ & 13 & 9 &  -- &  --   \\
\hline
$O_6$ & \bf 66 & 21  &  11 &  --  \\
\hline
$O_7$ & 158   & 90  &  17 &  12  \\
\hline\hline
\end{tabular}
\end{center}
\caption{Some bounds for $\alpha_k$ in odd graphs $O_{\ell}$ for $\ell=4,5,6,7$.}
\label{table5}
\end{table}

More generally,  \eqref{eq-k=1} and \eqref{eq-k=2} allow us to compute the bounds for $\alpha_1$ and $\alpha_2$ of every odd graph $O_{\ell}$, which turn out to be
\begin{align}
\alpha_1 & \le \frac{{2\ell\choose \ell}(\ell-1)}{4\ell-2}\sim \frac{2^{2\ell-2}}{\ell^{\frac{1}{2}}\sqrt{\pi}}, \label{alpha1odd}\\
\alpha_2 & \le \frac{{2\ell\choose \ell}(\ell-2)}{2(\ell+(-1)^\ell(\ell-2(-1^\ell))}\sim \frac{2^{2\ell-1}}{\ell^{\frac{3}{2}}\sqrt{\pi}}, \label{alpha2odd}
\end{align}
where we have indicated their asymptotic behaviour, when $\ell\rightarrow \infty$, by using the Stirling's formula. As a consequence, we have 
%the following result, partially answering an old question about the existence of perfect codes in odd graphs (see, for instance, Godsil \cite{g93}).
the known result that, when $\ell$ is odd, the odd graph $O_{\ell}$ has no $1$-perfect code.
Indeed, the existence of $1$-perfect code in $O_{\ell}$ requires that $\alpha_2=\frac{n}{\ell+1}=\frac{{2\ell\choose \ell}}{2(\ell+1)}$ (since all codewords must be mutually at distance $\ge 3$). However, when $\ell$ is odd, \eqref{alpha2odd} gives
$\alpha_2\le \frac{{2\ell\choose \ell}(\ell-2)}{2(\ell-1)(\ell+2)}<\frac{{2\ell\choose \ell}}{2(\ell+1)}$, a contradiction. (In fact, when $n$ is a power of two minus one $\frac{{2\ell\choose \ell}}{2(\ell+1)}$ is not an integer, which also prevents the existence of  a $1$-perfect code.)
Note that this result is in agreement with the fact that  a necessary condition for a regular graph to have a  $1$-perfect  code is the exitence of the eigenvalue $-1$, which is not present in $O_{\ell}$ when $\ell$ is odd (see Godsil \cite{g93}).

Finally, by using the same polynomial as in Subsection \ref{antipodal}, we have that the $(d-1)$-independence number of $O_{\ell}$, where $d-1=\ell-2$,  satisfies the bounds
$$
\alpha_{\ell-2}\le 1+m_1\frac{\pi_1}{\pi_0}=
\left\{ 
\begin{array}{ll}
2\ell-1, & \mbox{$\ell$ even},\\
2\ell-2, & \mbox{$\ell$ odd}.
\end{array}\right.
$$
For instance, for the Petersen graph $P=O_3$, this yields $\alpha_1\le 4$, as it is well-known.

\subsection*{Acknowledgments}
This research has been partially supported by AGAUR from the Catalan Government under project 2017SGR1087, and by MICINN from the Spanish Government under project PGC2018-095471-B-I00.

%\newpage

\end{document}